\def\cM{{\mathcal{M}}} \def\cN{{\mathcal{M}^{*}}}

\def\d{{\delta}}

  \def\({{\left(}} \def\){{\right)}}
  \def\[{{\left[}} \def\]{{\right]}}

\newcommand\tp{\operatorname{tp}}
\renewcommand\geq{\geqslant}
\renewcommand\leq{\leqslant}

\newcommand\std{\operatorname{st}}

\documentclass[11pt,draft]{amsart}

\usepackage{graphicx,amssymb,amsmath,verbatim,url,caption,amsthm,mathrsfs}
\usepackage[all]{xy}
\usepackage{stmaryrd}

\theoremstyle{plain}
\newtheorem{thm}{Theorem}

\newtheorem{cor}[thm]{Corollary}

\newtheorem{lem}[thm]{Lemma}

\newtheorem{prp}[thm]{Proposition}

\theoremstyle{definition}

\theoremstyle{remark}

\begin{document}

\title[A Simple Proof of the Marker-Steinhorn Theorem]{A Simple Proof of the Marker-Steinhorn Theorem For Expansions of Ordered Abelian Groups}
\author{Erik Walsberg}
\address{University of California, Los Angeles\\
Los Angeles, CA 90955\\
U.S.A.}
\email {erikw@math.ucla.edu}
\date{\today}

\begin{abstract} We give a short and self-contained proof of the Marker-Steinhorn Theorem
for o-minimal expansions of ordered groups, based on an analysis of linear orders definable in such structures. \end{abstract}

\maketitle

\section{Introduction}
\noindent
Let $\cM=(M,{\leqslant},\dots)$ be a dense linear order without endpoints, possibly with additional structure, in the language $\mathcal{L}$. 
A type $p(x)$ over $M$ is said to be {\it definable}\/ if for every $\mathcal{L}$-formula
$\delta=\delta(x,y)$ in the (object) variables $x=(x_1,\dots,x_m)$ and (parameter) variables $y=(y_1,\dots,y_n)$,
there is a  defining formula  for the restriction $p\!\upharpoonright\!\delta$ of $p$ to $\delta$, i.e., a formula $\phi(y)$,
possibly with parameters from $M$, such that $\delta(x,b)\in p \Longleftrightarrow \cM\models \phi(b)$, for all $b\in M^{n}$. 
The Marker-Steinhorn Theorem alluded to in the title of this note gives a condition for certain types over $M$ to be definable, provided that $\cM$ is o-minimal.

To explain it, we first recall that a set $C\subseteq M$ is said to be a {\it cut}\/ in $\cM$ if whenever $c\in C$, then $(-\infty,c):=\{a\in M:a < c\}$ is contained in $C$. Let $\d(x,y)$ be the formula $x>y$ (in the language of $\cM$). It is well known that cuts in $\cM$ correspond in a one-to-one way to complete $\d$-types over $M$, where to the cut $C$ in $\cM$ we  associate the complete $\d$-type 
$$p_C(x):=\{\d(x,b):b\!\in\!C\}\cup\{\neg\d(x,b):b\in M\setminus C\}$$ 
over $M$. 
The $\d$-type $p_C$ is definable if and only if the cut $C$ in $\cM$ is definable (as a subset of $M$). 
If $C$ is of the form $(-\infty,c]:=\{a\in M:a\leqslant c\}$ ($c\in M$) or $(-\infty,c)$ ($c\in M\cup\{\pm\infty\}$), then $C$ clearly is definable. Cuts of this form are said to be {\it rational.}\/ The structure $\cM$ is definably connected if and only if all definable cuts are rational. If $(M,{\leqslant})=(\mathbb R,{\leqslant})$ is the  real line with its usual ordering, then all cuts in $\cM$ are rational. This can be used to define the standard part map for elementary extensions. That is, if $(M,{\leqslant})=(\mathbb R,{\leqslant})$ and $\cM\preceq\cN=(M^*,{\leqslant},\dots)$, then we can define a map   
$$b\mapsto \sup\{a\in M: a\leqslant b\}\colon M^*\cup\{\pm\infty\}\longrightarrow M\cup\{\pm\infty\},$$ 
where we declare $\sup \emptyset:=-\infty$ and $\sup M:=+\infty$. To generalize this, we say that an elementary extension $\cM\preceq\cN$ is {\it tame}\/ if for every $a\in M^{*}$ the cut $\{b\in M:b\leqslant a\}$ is rational.
(Thus if  $(M,{\leqslant})$ is the usual ordered set of reals, then every elementary extension of $\cM$ is tame.)
 We can then define a standard part map in the same way.
\\

Now $\cM$ is o-minimal if and only if every $1$-type over $M$ is determined by its restriction to $\d$, in which case a $1$-type over $M$ is definable exactly when the associated cut in $\cM$ is rational. It trivially follows that $\cM\preceq\cN$ is tame if and only if for every $a\in M^*$, the type $\tp(a|M)$ is definable. Marker and Steinhorn~\cite{Marker-Steinhorn} generalized this to show that if $\cM$ is o-minimal and $\cM\preceq\cN$ is tame then for every $a\in (M^*)^m$, the type $\tp(a|M)$ is definable. In particular, if $\cM$ is a structure on the real line, then every type over $M$ is definable. 
See \cite{van den Dries 1} for a survey of geometric applications of this very useful result.
The original proof of Marker and Steinhorn uses a complicated inductive proof. Tressl~\cite{Tressl} proved the Marker-Steinhorn theorem for o-minimal expansions of real closed fields with a short and clever argument. His proof gives little idea as to the form of the defining formulas of a type. Chernikov and Simon have given a proof using NIP-theoretic machinery~\cite{CS}. We give a short proof of the Marker-Steinhorn Theorem for o-minimal expansions of ordered groups. 
 The crucial idea behind our proof is to reduce the analysis of $n$-types to an analysis of cuts in definable linear orderings.
 Our main tool is Proposition~\ref{prp:Janak}, a result about linear orders definable in o-minimal structures admitting elimination of imaginaries. This result is essentially due to Ramakrishnan~\cite{Ramakrishnan}, which is closely related to earlier work of Onshuus-Steinhorn~\cite{Onshuus-Steinhorn}. For the sake of completeness we provide a proof.
\\

By carefully tracking the parameters used to define the type, we actually obtain a uniform version of the Marker-Steinhorn theorem.  The  pair $(\cN,\cM)$ is the structure that consists of $\cN$ together with a unary predicate for the underlying set of $\cM$ and a unary function symbol for the restriction of the standard part map $\std$ to the convex hull of $M$ in $M^*$. The expanded language is called $\mathcal{L}^{*}$. 
We denote by $\mathcal{L}(M)$ the expansion of $\mathcal L$ by constant symbols naming each
element of $M$, and similarly with $\mathcal{L}^*$ in place of $\mathcal L$.
We show that if $\delta(x,y)$ is an  $\mathcal{L}$-formula then there is an $\mathcal{L}(M)$-formula $\phi(z,y)$ and an $\mathcal{L}^{*}(M)$-definable map $\Omega$, taking values in a cartesian power of $M$, such that for any tuples $a$ in $M^*$ and $b$ in $M$ of appropriate lengths,
$$
\cN\models\delta(a,b)\quad\Longleftrightarrow\quad \cM\models\phi\big(\Omega(a),b\big).
$$
We will prove this by induction on the length of $a$.
See Proposition~\ref{prp:main prp} below for a precise statement and the proof.

\subsection*{Conventions}
We let $m$,~$n$ and $k$ range over the set $\mathbb N=\{0,1,2,\dots\}$ of natural numbers. 
Given sets $A$, $B$ and $C\subseteq A\times B$, as well as $a\in A$ and $b\in B$, we let 
$$C_{a}=\{b\in B:(a,b)\in C\}, \qquad C^{b}=\{a\in A:(a,b)\in C\}.$$ 
Throughout the paper, $\cM$ is an o-minimal expansion of a dense linear order without endpoints, admitting elimination of imaginaries, and $\cM\preceq\cN$ is a tame extension. 
If $A\subseteq M^{m}$ is a definable set, then $A^{*}$ denotes the subset of $(M^*)^m$ defined  in $\cN$ by the same formula. (Since $\cM\preceq\cM^*$, this does not depend on the choice of defining formula.) Similarly, if $f\colon A\to M^n$, $A\subseteq M^m$, is a definable map, then $f^*\colon A^*\to (M^*)^n$ denotes the map whose graph is defined in $\cN$ by the same formula as the graph of $f$.
Unless said otherwise, ``definable'' means ``definable, possibly with parameters,'' and
the adjective ``definable''  applied to subsets of $M^m$ or maps $A\to M^n$, $A\subseteq M^m$, will mean ``definable in $\cM$.'' Let $A,B\subseteq M^m$ be definable.
By $\dim(A)$ we denote the usual o-minimal dimension of~$A$. 
If $A\subseteq B$, then we say that 
$A$ is  {\it almost all}\/  of $B$  if $\dim(B\setminus A)<\dim(B)$, and we say that a property of elements of $B$ is true of {\it almost all $b\in B$}\/ if it holds on a definable subset of $B$ which is almost all of $B$. Let $\sim$ be a definable equivalence relation on $A$. Then 
for $a\in A$ we let $[a]_{\sim}$ denote the $\sim$-class of $a$, and
we let 
$$A/{\sim}:=\big\{[a]_{\sim}:a\in A\big\}$$ 
be the set of equivalence classes of $\sim$. We tacitly assume that
(by elimination of imaginaries) we are given a definable set $S\subseteq A$ of representatives of $\sim$, and identify $S$ with $A/{\sim}$. The basic facts about o-minimal structures that we use can be found in~\cite{van den Dries 2}. 
If $\mathcal M$ expands an ordered abelian group,
given a bounded definable $A\subseteq M$ we let $\mu(A)$ be the sum of the lengths of the components of $A$. If $A\subseteq M^{m}\times M$ is such that every $A_x$ is bounded then there is a definable $f\colon M^{m}\longrightarrow M$ such that $f(x)=\mu(A_x)$. We call $\mu(A)$ the {\it measure}\/ of $A$. (Indeed, $\mu$ is a finitely additive measure on the collection of bounded definable subsets of $M$.)

\subsection*{Acknowledgments}
We thank Matthias Aschenbrenner for suggesting the topic, for many useful discussions on the topic,  and for finding a serious gap in the first version of the proof. We also thank David Marker for his comments on an earlier version of the proof.

\section{Definable Linear Orders}

\noindent
In this section we establish a key result about definable linear orders in~$\cM$. As mentioned earlier, this fact is a very weak version of a result due to Ramakrishnan (related to earlier work of Onshuus-Steinhorn). It can fairly easily be proved directly; for sake of completeness, and since we also need to investigate the uniformities in the construction, we include a proof.

We fix a definable  linear order $(P,\leqslant_P)$; i.e., $P$ is a  definable subset of~$M^m$, for some~$m$,
and $\leqslant_P$ is a definable binary relation on $P$ which is a linear ordering (possibly with endpoints). Sometimes we suppress $\leqslant_P$ from the notation. 
We let $a$, $b$, $c$ range over $P$.
A map $\rho \colon P\longrightarrow Q$, where $Q$ is a definable linear order, is said to be monotone if $a\leqslant_P b\Rightarrow \rho(a)\leqslant_Q\rho(b)$, for all $a$,~$b$.

\begin{prp}\label{prp:Janak}
Suppose $l=\dim(P)\geqslant 2$. Then there is a definable linear order $(Q,\leqslant_{Q})$ such that $\dim(Q)=l-1$ and a definable surjective monotone map $\rho\colon P\longrightarrow Q$ all of whose fibers have dimension at most $1$.
\end{prp}

We first reduce the proof of this proposition to constructing a certain definable equivalence relation on $P$. Suppose that $\sim$ is a definable equivalence relation on $P$ whose equivalence classes are convex (with respect to $\leqslant_P$), have dimension at most~$1$, and for almost all $a$, $[a]_{\sim}$ is infinite. The first condition ensures that the linear order on $P$ pushes forward to a definable linear order $\leqslant_Q$ on $Q:=P/{\sim}$, so that the quotient map  $\rho\colon P\longrightarrow P/{\sim}$ becomes monotone. The third condition ensures that $\dim(P/{\sim})=l-1$. Then $\rho\colon P\longrightarrow Q$   satisfies the conditions of Proposition~\ref{prp:Janak}.
\\
\\
If all intervals 
$$(a,b)_P := \{c: a <_P c <_Pb\}\qquad (a<_Pb)$$
in $P$ are infinite then a convex subset of $P$ is infinite if it has at least two elements. The next lemma allows us to assume that all intervals in $P$ are infinite.

\begin{lem}
There is a definable surjective monotone map $P\longrightarrow R$ to a definable linear order $R$ in which all intervals are infinite and $\dim(P)=\dim(R)$.
\end{lem}
\begin{proof}
Take $N\in\mathbb N$ such that if $(a,b)_P$ is finite then $|(a,b)_P|<N$. Every finite interval in $P$ is contained in a maximal finite interval. Define $a\sim_{f}b$ if $a$ and $b$ are contained in the same maximal finite interval. This is a definable equivalence relation on $P$ with convex equivalence classes. For each $a$, $[a]_{\sim_{f}}$ is a finite interval and so has cardinality strictly less than~$N$. Let $R=P/{\sim_{f}}$, equipped with the definable linear order making the natural projection $P\longrightarrow R$ monotone. As the quotient map is finite-to-one, $\dim(R)=\dim(P)$. Suppose $[a]_{\sim_{f}}$,~$[b]_{\sim_{f}}$ are distinct elements of $R$ with $a<_Pb$. Then $(a,b)_P$ is infinite and so contains infinitely many $\sim_{f}$-classes. Thus $([a]_{\sim_{f}},[b]_{\sim_{f}})_R$ is infinite.
\end{proof}

If $P\longrightarrow R$ is as in the previous lemma, and if we have a map $\rho\colon R\longrightarrow Q$ which satisfies the conditions on $\rho$ in Proposition~\ref{prp:Janak} with $R$  replaced by~$P$, then the composition of $\rho$ with the map $P\longrightarrow R$  satisfies the conditions on~$\rho$ in Proposition~\ref{prp:Janak}. {\it We henceforth assume that all intervals in $P$ are infinite.}\/  
\\
\\
We now define the required equivalence relation. Let $d\leq l$ be a natural number.
We say that $a\sim_{d} b$ if $\dim\,(a,b)_P<d$.  It is very easy to see that $\sim_{d}$ is a definable equivalence relation on $P$, and even easier to see that its equivalence classes are convex. Lemma~\ref{lem:contains l-dim interval} below will be used to show that $\dim\, [a]_{\sim_{d}}<d$ for all $a$. It is more difficult to show that almost all $[a]_{\sim_{d}}$ are infinite.  Our desired equivalence relation is~$\sim_{1}$; we will show that the quotient map $P\longrightarrow P/{\sim_{1}}$ satisfies the conditions of Proposition~\ref{prp:Janak}. This proof uses Lemma~\ref{lem:one-dim}.

\begin{lem}\label{lem:contains l-dim interval}
The ordered set $P$ contains an $l$-dimensional interval.
\end{lem}

\begin{proof}
Let $D\subseteq P^{3}$ be the set of triples in $P^3$ with pairwise distinct components. 
Then clearly $\dim(D)=3l$.
Let $E$ be the set of $(a,b,c)\in D$ such that $a<_P b <_P c$. Let $f\colon D\longrightarrow D$ be the map given by $f(a_{1},a_{2},a_{3})=(a_{\sigma(1)},a_{\sigma(2)},a_{\sigma(3)})$, 
where $\sigma$ is the permutation of $\{1,2,3\}$ such that
$a_{\sigma(1)}<_Pa_{\sigma(2)}<_Pa_{\sigma(3)}$. Clearly $f$ is finite-to-one, and $E=f(D)$, so $\dim(E)=3l$. 
Since $$E=\bigcup_{a,b} \,\{a\}\times (a,b)_P \times\{b\},$$
there are $a$, $b$ such that $\dim\,(a,b)_P=l$.
\end{proof}

\begin{cor}
$\dim\,[a]_{\sim_{d}}<d$ for all $a$.
\end{cor}
\begin{proof}
The set
$[a]_{\sim_{d}}$ with the order induced by $\leqslant_P$ is a definable linear order. Applying the lemma above, there exists $b_{1},b_{2}\in[a]_{\sim_{d}}$ such that $\dim\,(b_{1},b_{2})_P=\dim\,[a]_{\sim_{d}}$. As $b_{1}\sim_{d}b_{2}$, $\dim\,(b_{1},b_{2})_P<d$.
\end{proof}

\begin{lem}\label{lem:one-dim}
The set $C$ consisting of all $a$ such that $(-\infty,a]_P$ and $[a,+\infty)_P$ are both closed 
\textup{(}in $M^m$\textup{)}
is at most one-dimensional.
\end{lem}

\begin{proof}
Let $d=\dim(C)$ and let
$C'\subseteq C$ be a $d$-dimensional cell.  Let $a_{1},b,a_{2}$ be distinct elements of $C'$ with $a_{1}<b<a_{2}$.
Then $a_{1}\in (-\infty,b)_P\cap C'$ and $a_{2}\in(b,+\infty)_P\cap C'$. Thus $(-\infty,b)_P\cap C'$ and $(b,+\infty)_P\cap C'$ form a nontrivial partition of  $C'\setminus\{b\}$ into disjoint closed sets.
Thus $C'\setminus\{b\}$ is not definably connected, and so $d=\dim(C')\leqslant 1$.
\end{proof}

\begin{lem}
$[a]_{\sim_l}$ is infinite,  for almost all $a$.
\end{lem}
\begin{proof}
Let $O$ be the set of $(a,b)\in P\times P$ such that $a>_P b$. Note that $O_{c}=(-\infty,c)_P$ and $O^{c}=(c,+\infty)_P$, for each $c$. We let $D$ be the boundary of $O$ in $P\times P$. As $\dim(D)<2l$, for almost all $c\in P$ we have $\dim(D_{c}),\dim(D^{c})<l$. Let $E$ be the set of $c$ such that $\dim(D_{c})\geqslant l$ or $\dim(D^{c})\geqslant l$.

Note that $[a]_{\sim_{l}}$ is finite if and only if it equals $\{a\}$. Let $A$ be the set of $a$ such that 
$[a]_{\sim_{l}}=\{a\}$.
Suppose that  $a<_{P}c$ and $c$ is in the closure of $(-\infty,a)_{P}$. Let $b\in(a,c)_{P}$. So $(c,b)\in O$, and $(c,b)$ is a limit point of $(-\infty,a)\times\{b\}\subseteq[P\times P]\setminus O$. Hence $(c,b)\in D$. This holds for any element of $(a,c)_{P}$, so $\{c\}\times(a,c)_{P}\subseteq D$. Hence $(a,c)_{P}\subseteq D_{c}$, and as $c\in A$, $\dim\,(a,c)_{P}=l$, so $\dim\,D_{c}\geqslant l$. Thus $c\in E$. An analogous argument shows that if there is an $a$ such that $a>_{P}c$ and $c$ is in the closure of $(a,+\infty)_{P}$, then $c\in E$. It follows from what we have shown that if $c_1,c_2\in A\setminus E$ and $c_1<_{P}c_2$ then $c_1$ is not in the closure of $(c_2,+\infty)_{P}$ and $c_2$ is not in the closure of $(-\infty,c_1)_{P}$. Consider $A\setminus E$ as a definable linear order with the order induced from $P$. For all $c\in A\setminus E$, both $(-\infty,c]_{A\setminus E}$ and $[c,+\infty)_{A\setminus E}$ are closed. From Lemma~\ref{lem:one-dim} we obtain $\dim(A\setminus E)=1$. So either $\dim(A)=1$ or $\dim(A)=\dim(E)<l$. In either case $\dim(A)<l$.
\end{proof}

With the following lemma we now finish the proof of Proposition~\ref{prp:Janak}:

\begin{lem}
$[a]_{\sim_{1}}$ is infinite, for almost all $a$.
\end{lem}

\begin{proof} 
We show this by induction on $l=\dim(P)$. If $l=1$, then this is trivially true. Suppose this statement holds 
for all smaller values of $l$.
For almost all $a$, $[a]_{\sim_{l}}$ is infinite, by the previous lemma. As $\dim\,[a]_{\sim_{l}}<l$ for almost all $b\sim_{l}a$, there are infinitely many $c\in[a]_{\sim_{l}}$ such that $c\sim_{1} b$. The fiber lemma for o-minimal dimension now implies that $[a]_{\sim_{1}}$ is infinite for almost all~$a$.
\end{proof}

Note that these constructions are done uniformly in the parameters defining $(P,{\leqslant}_{P})$. Namely, if $P\subseteq M^{k}\times M^{m}$ and ${\leq_P}\subseteq M^{k}\times (M^{m}\times M^{m})$ are definable sets such that for each $a\in M^{k}$, ${\leq_{P_a}}:=(\leq_{P})_a$ is a linear order on $P_{a}$, then there are definable sets $Q\subseteq M^{k}\times M^{n}$, ${\leq_Q}\subseteq M^{k}\times(M^{n}\times M^{n})$ and $R\subseteq M^{k}\times (M^{m}\times M^{n})$ such that for each $a\in M^{k}$ with $\dim(P_a)\geq 2$, ${\leq_{Q_a}}:=(\leq_{Q})_a$ is a linear order on $Q_{a}$ and $R_{a}$ is the graph
of a monotone map  $(P_{a},{\leq_{P_a}}) \to (Q_{a},{\leq_{Q_a}})$ satisfying the conditions of Proposition~\ref{prp:Janak}.

\section{Rational Cuts in Definable Linear Orders}

\noindent From now on until the end of the paper we assume that $\cM$ expands an ordered abelian group.
As in the previous section, we let $(P,{\leqslant_P})$ be a definable linear order.
We now give an application of Proposition~\ref{prp:Janak} used in our proof of the Marker-Steinhorn Theorem in the next section. Recall that we assume $P\subseteq M^m$.

\begin{prp}\label{prp:Janak, consequence}
If $V\subseteq P^{*}$ is definable in $\cM^{*}$ and $W=V\cap P$ is a cut in~$P$, then $W$ is definable in $\cM$.
\end{prp}

The proof of this proposition is the most difficult part of this paper. The difficulty largely lies in the fact that $V$ is not assumed to be a cut in $P^*$. If $V$ was a cut, then we could try to prove the result in the following way: argue by induction on $\dim (P)$, let $\rho\colon P\longrightarrow Q$ be the map given by Proposition~\ref{prp:Janak}, let $B\subseteq Q^*$ be the set of $q$ such that $\rho^{-1}(q)\subseteq V$, argue inductively that $B\cap Q$ is $\cM$-definable, and use this to show that $W$ is $\cM$-definable. It is natural to try to apply this arguement to our situation by replacing $V$ with its convex hull $V'$ in  $Q$. However $W$ can be a proper subset of $V'\cap P$. For example let $P=(M,\leqslant$), let $t$ be an element of $M^*$ larger then every element of $M$, and let $V=(0,1)\cup \{t\}$.

\medskip
\noindent
In the proof of Proposition~\ref{prp:Janak, consequence} we also need the following two lemmas. The first is the base case of the Marker-Steinhorn theorem.

\begin{lem}\label{one-dim}
Let $A\subseteq M^{m}$ be a definable one-dimensional subset of~$M^{m}$, and let $B\subseteq (M^*)^{m}$ be definable in $\cN$. Then $B\cap A$ is definable in $\cM$.
\end{lem}

\begin{proof}
By Cell Decomposition, $A$ is the union of finitely many sets of the form $f(M)$, where
$f\colon M\longrightarrow A$ is a definable map.
We may thus reduce to the case that $A$ itself is of this form.
It suffices to show that $f^{-1}(B\cap A)=(f^*)^{-1}(B)\cap M$ is definable. So we may assume that $m=1$ and $A=M$.
Then $B$ is a boolean combination of rays of the type $(-\infty,b)_{M^*}$ or $(-\infty,b]_{M^*}$, where $b\in M^*$.
Let $b\in M^*$; if $b>M$, then $M\cap (-\infty,b)_{M^*}=M\cap (-\infty,b]_{M^*}=M$;
otherwise, $M\cap (-\infty,b)_{M^*}$ and $M\cap (-\infty,b]_{M^*}$ each equal one of
$(-\infty,\std(b))_{M}$ or $(-\infty,\std(b)]_{M}$. 
\end{proof}

\begin{lem}\label{lem:measures}
Let $A\subseteq M$ be bounded, infinite and definable, and let $B\subseteq A^{*}$ be definable in $\cN$. If $A\subseteq B$, then $\std(\mu(B))=\mu(A)>0$. If $A\cap B=\emptyset$ then $\std(\mu(B))=0$.
\end{lem}

\begin{proof}
Let $c<d$ be elements of $M^{*}$ contained in the convex hull of $M$. If $\std(d-c)>0$ then $(c,d)_{M^*}$ must contain infinitely many elements of $M$. Therefore if $\std(\mu(B))>0$, then 
$B$ contains infinitely many elements of $A$ (as then $B$ contains an interval whose length is not infinitesimal); so $A\cap B\neq\emptyset$. If $\std(\mu(B))<\mu(A)=\mu(A^{*})$ then $\std(\mu(A^{*}\setminus B))>0$, so as before $A^{*}\setminus B$ contains infinitely many points in $A$, therefore $A$ is not a subset of~$B$.
\end{proof}

\begin{proof}[Proof of Proposition~\ref{prp:Janak, consequence}]
We use induction on $l=\dim(P)$.  If $l=1$, then this is a special case of the preceding Lemma~\ref{one-dim}. Suppose that $l\geqslant 2$. Take $(Q,\leqslant_{Q})$ and $\rho\colon P\longrightarrow Q$ as in Proposition~\ref{prp:Janak}. We fix a positive element $1$ of $M$ and identify $\mathbb Q$ with its image under the embedding $\mathbb Q\to M$ of (additive) ordered abelian groups which sends $1\in\mathbb Q$ to $1\in M$. 
We shall specify an integer $N\geq 1$ and a definable injective map 
$$\iota\colon P\longrightarrow Q\times M\times\{1,\dots,m\}\times\{1,\dots,N\} \subseteq Q\times M\times M\times M$$
with the property that  $\iota(p)=(\rho(p),\dots)$ for each $p\in P$. 
We let $i$ range over  $\{1,\dots, m\}$, and for each $i$ we let $\pi_{i}\colon P\longrightarrow M$ be the restriction to $P$ of the projection $M^m\to M$ onto the $i$th coordinate. 
For each $q\in Q$ define $P_{q}^{i}$  inductively as the set of $a\in\rho^{-1}(q)\setminus (P_{q}^{1}\cup\dots\cup P_{q}^{i-1})$ such that there are only finitely many
$b\in\rho^{-1}(q)$ with $\pi_{i}(a)=\pi_{i}(b)$. For each $q\in Q$, $\rho^{-1}(q)$ is then the disjoint union of the~$P_{q}^{i}$. Let $N\in\mathbb N$ be such that for all $q$,~$i$, the fibers of $\pi_{i}|_{P_{q}^{i}}$ have cardinality bounded by $N$.
If $p\in P_{\rho(p)}^{i}$ is the $j$th element of $\pi_{i}^{-1}(\rho(p))\cap P_{\rho(p)}^{i}$
 in the lexiographic order induced from $M^{m}$, then we set $$\iota(p)=(\rho(p),\pi_{i}(p),i,j).$$
Below, we let $j$ range over $\{1,\dots,N\}$.

Let now $V\subseteq P^{*}$ be definable in $\cM^{*}$ such that $W=V\cap P$ is a cut in~$P$.
As $\rho$ is monotone, $\rho(W)$ is  a cut in $Q$. We construct a set
$B\subseteq Q^{*}$, definable in $\cM^{*}$, such that $B\cap Q=\rho(W)$. (It will then follow from the inductive hypothesis that $\rho(W)$ is definable.) It is easily seen that if $q$ is a non-maximal element of $\rho(
W)$ then $\rho^{-1}(q)$ is contained in $W$. It is also easily seen that if $q\in Q$ is not in $\rho(W)$ then $\rho^{-1}(q)$ is disjoint from $W$.
For $q\in Q^*$ we define  
$$P(q):=\iota^*\big((\rho^*)^{-1}(q)\big),\qquad W(q):=\iota\big(W\cap(\rho^*)^{-1}(q)\big).$$ 
and $$V(q)=\iota^{*}\big(V\cap (\rho^*)^{-1}(q)\big),$$
so that $V(q)\cap (Q\times M^3)=W(q)$ if $q\in Q$. 
Again, for all $q\in Q$, if $q$ is  a non-maximal element of $\rho(W)$ then $P(q)\subseteq W(q)$, and if
$q\notin\rho(W)$ then $W(q)=\emptyset$.

For $q\in Q^*$ let $P(q,i,j)$ be the set of $s\in M^*$ such that $(q,s,i,j)\in P(q)$, and define   $V(q,i,j)\subseteq M^*$ likewise.
Now we list some consequences of Lemma~\ref{lem:measures}. For this, let $q\in Q$ and  $c,d\in M^*$ with $c<d$.
If  $P(q)\subseteq W(q)$, then:
\begin{enumerate}
\item If $(c,+\infty)_{M^*}$ is the interior of a component of $P(q,i,j)$ then  $$\std\big(\mu(V(q,i,j)\cap[c,c+1]_{M^*})\big)=1.$$ 
\item If $(-\infty,c)_{M^*}$ is the interior of a component of $P(q,i,j)$ then $$\std\big(\mu(V(q,i,j)\cap [c-1,c]_{M^*})\big)=1.$$ 
\item If $(c,d)_{M^*}$ is the interior of a component of $P(q,i,j)$ then $$\std\big(\mu(V(q,i,j)\cap[c,d]_{M^*})\big)=d-c.$$
\end{enumerate}
On the other hand, if  $W(q)=\emptyset$, then in each of the preceding cases the standard part of the measure of the intersection of $V(q,i,j)$ with the appropriate segment in $M^*$ is zero. Let now $\Lambda\in Q$ be the maximal element of $\rho(W)$ if this exists, and some fixed element of $Q$ otherwise. We let $B$ be the set of $q\in Q^{*}$ such that for all $i$, $j$ and all $c<d$ in $M^*$,
\begin{enumerate}
\item if $(c,+\infty)_{M^*}$ is the interior of a component of $P(q,i,j)$ then $$\mu\big(V(q,i,j)\cap [c,c+1]_{M^*}\big)<\frac{1}{2};$$
\item if $(-\infty,c)_{M^*}$ is the interior of a component of $P(q,i,j)$ then $$\mu\big(V(q,i,j)\cap [c,c-1]_{M^*}\big)<\frac{1}{2};$$
\item if $(c, d)_{M^*}$ is the interior of a component of $P(q,i,j)$ then $$\mu\big(V(q,i,j)\cap [c,d]_{M^*}\big)<\frac{1}{2}(d-c).$$
\end{enumerate} 
The set $B$ is definable in $\cN$, and $B\cap Q$ is the set of all nonmaximal elements of $\rho(W)$, possibly together with $\Lambda$. This is a cut in $Q$. By induction, $B\cap Q$ is definable in $\cM$. Let $p\in P$. If $\rho(W)$ has a maximal element then $p$ is in $W$ if $\rho(p)<\Lambda$ or if $p\in W\cap \rho^{-1}(\Lambda)$. By Lemma~\ref{one-dim}, $W\cap\rho^{-1}(\Lambda)$ is definable in $\cM$. If $\rho(W)$ does not have a maximal element then $p\in W$ if and only if $\rho(p)\in B$.
\end{proof}

By carefully keeping track of the parameters used in the proof of  Lemma~\ref{one-dim}, we see that we have in fact proven the following uniform version of the lemma, which also provides the base case of the uniform Marker-Steinhorn Theorem.

\begin{lem}\label{one-dim, uniform}
Let $A\subseteq M^k \times M^m$ be definable with $\dim(A_x)=1$ for every $x\in M^k$, and let $B\subseteq (M^{*})^{j}\times(M^*)^m$ be definable in $\cN$. Then there is a definable $E\subseteq M^l \times M^m$, for some $l$, and a  map $\Omega\colon M^k\times (M^{*})^{j}\to M^l$, definable in the $\mathcal L^*$-structure $(\cN,\cM)$, such that $A_x\cap B_{a}=E_{\Omega(x,a)}$.
\end{lem}

Similarly, by carefully keeping track of the parameters used to define $B$ and $W$ and strengthening the inductive assumption in the natural way, we can see that we have in fact proven a uniform version of Proposition \ref{prp:Janak, consequence}: $\iota$ can be defined uniformly in the same way as $\rho$; $B$ can be defined uniformly from $W$; and if $W\subseteq (M^{*})^{k}\times P^{*}$ is definable in $\cN$ then the map $\Lambda\colon (M^{*})^{k}\longrightarrow Q$ that takes $a$ to the maximum of $\rho(W_{a}\cap P)$ if such exists and to some fixed element of $Q$ otherwise, is definable in $(\cN,\cM)$.

\begin{prp}\label{prp:Janak, consequence, uniform}
Let $P\subseteq M^{l}\times M^{m}$ and $\leq$ be a subset of $M^{l}\times [M^{m}\times M^{m}]$ such that for every $a\in M^{l}$, $\leq_{a}$ is a linear order on $P_{a}$. Let $V\subseteq (M^*)^{k}\times P^{*}$ be definable in $\cM^{*}$. Then there is some $j$ and a map $\Omega\colon (M^*)^k\times M^{l}\longrightarrow M^{j}$,   definable in the $\mathcal L^*$-structure $(\cN,\cM)$,  and a definable $W\subseteq M^{j}\times M^{m}$ such that for each $x\in (M^*)^{k}$ and $a\in M^{l}$, if $(V_{x}\cap P)_{a}$ is a cut in $P_{a}$ then $(V_{x}\cap P)_{a}=W_{\Omega(x,a)}$.
\end{prp}

We remark that the use of the function $\iota$ in the proof of Proposition~\ref{prp:Janak, consequence} may be avoided by using Ramakrishnan's theorem \cite{Ramakrishnan} on embedding definable linear orders into lexicographic orders. 
Moreover, the only point in our proof of the Marker-Steinhorn Theorem where we need to assume that $\cM$ expands an ordered abelian group is in Proposition~\ref{prp:Janak, consequence}.

\section{Proof of the Marker-Steinhorn Theorem}
\noindent
We now prove the uniform Marker-Steinhorn Theorem. Recall our standing assumption that $\cM\preceq\cN$ is a tame extension.

\begin{prp}\label{prp:main prp}
Let $\delta(x,y)$ be an $\mathcal{L}$-formula, where $x=(x_1,\ldots,x_m)$ and  $y=(y_1,\ldots, y_n)$. Then there is an $\mathcal{L}(M)$-formula $\phi(z,w)$, where $z=(z_1,\dots,z_k)$, and a 
map $\Omega\colon (M^*)^m\to M^k$, definable in the $\mathcal L^*$-structure $(\cN,\cM)$,
such that for all $a\in (M^*)^m$, $b\in M^{n}$:
$$\cM^*\models\delta(a,b)\quad\Longleftrightarrow\quad\cM\models \phi(\Omega(a),b).$$
\end{prp}
We use induction on $m$.  Lemma~\ref{one-dim, uniform} treats the base case $m=1$. Suppose that $m\geq 2$. Let $\hat{a}=(a_{1},\ldots,a_{m-1})$; inductively, $\tp(\hat{a}|M)$ is definable. We construct a defining formula for the restriction $\tp(a|M)\!\upharpoonright\!\d$ of $\tp(a|M)$ to~$\delta$. It is a direct consequence of the Cell Decomposition Theorem that $\delta(x;y)$ is a boolean combination of formulas $\delta_i(x;y)$
such that the set of tuples $(a,b)=(\hat{a},a_m,b)\in (M^*)^{m-1}\times M^*\times (M^*)^{n}$ 
defined by $\delta_i$ in $\cN$  has one of the following forms:
\begin{enumerate}
\item $(\hat{a},b)\in X^*$ and $a_m\geqslant f^*(\hat{a},b)$,
\item $(\hat{a},b)\in X^*$ and $a_m\leqslant f^*(\hat{a},b)$,
\item $(\hat{a},b)\in X^*$,
\end{enumerate}
where $X\subseteq M^{m+n-1}$ and $f\colon M^{m+n-1}\longrightarrow M$ is definable. The defining formula of $\tp(a|M)\!\upharpoonright\!\delta$ is the corresponding boolean combination of the defining formulas of $\tp(a|M)\!\upharpoonright\!\delta_i$. We therefore assume that $\delta$ is of one these forms. The last case is rendered trivial by the inductive assumption. We now suppose that $\delta$ is of the first form. Thus
$$\cM^*\models\d(a,b)\Longleftrightarrow \text{$(\hat{a},b)\in X^*$ and $f^*(\hat{a},b)\leqslant a_{m}$.}$$
By the induction hypothesis we take a definable $B\subseteq M^{k}\times M^n$, for some~$k$, and map  $\Omega_1\colon (M^*)^{m-1}\longrightarrow M^k$, definable in
the pair $(\cM^*,\cM)$, such that for
$\hat{a}\in (M^{*})^{m-1}$ and $b\in M^n$ we have
$$(\hat{a},b)\in X^* \quad\Longleftrightarrow\quad b\in B_{\Omega_1(\hat{a})}.$$
For
$\hat{a}\in (M^{*})^{m-1}$,
$b_{1},b_{2}\in M^n$ with $(\hat{a},b_i)\in X^*$ ($i=1,2$), we define 
$$b_{1}\lesssim_{\hat{a}} b_{2}\quad :\Longleftrightarrow\quad f^*(\hat{a},b_1)\leqslant f^*(\hat{a},b_2).$$ 
Again, the inductive hypothesis gives a definable $C\subseteq M^l\times (M^n\times M^n)$ and a map $\Omega_2\colon (M^*)^{m-1}\longrightarrow M^l$ which is definable in $(\cN,\cM)$ and such that $b_1\lesssim_{\hat{a}} b_2$ if and only if $(b_1,b_2)\in C_{\Omega_2(\hat{a})}$, for all $\hat{a}\in (M^*)^{m-1}$ and $b_1,b_2\in M^n$. 
It is easy to check that each $\lesssim_{\hat{a}}$ is a quasi-order on $M^{n}$ in which any two elements are comparable. For $b_{1},b_{2}\in M^{n}$  set 
\begin{align*}
b_{1}\sim_{\hat{a}} b_{2}&\quad :\Longleftrightarrow\quad b_{1}\lesssim_{\hat{a}} b_{2}\text{ and }b_{2}\lesssim_{\hat{a}} b_{1} \\
&\quad\ \Longleftrightarrow \quad f^*(\hat{a},b_{1})=f^*(\hat{a},b_{2})\\
&\quad\ \Longleftrightarrow \quad (b_1,b_2), (b_2,b_1)\in C_{\Omega_2(\hat{a})}.
\end{align*}
This is a definable equivalence relation on $M^{n}$. Let 
$$C':=\big\{ (b,b_1,b_2)\in M^l\times [M^n\times M^n]: (b_1,b_2),(b_2,b_1)\in C_b\big\}.$$ 
If $b\in M^l$ is of the form $\Omega_2(\hat{a})$ then $C'_b$ is a definable equivalence relation on~$M^n$. By uniform elimination of imaginaries let $A\subseteq  M^{l}\times M^{n}$ be definable such that for all $b\in M^l$ we have $A_b=M^n/C'_b$ whenever $C'_b$ a definable equivalence relation, and $A_b=\emptyset$  otherwise. So we have that $A_{\Omega_{2}(\hat{a})}=M^{n}/{\sim_{\hat{a}}}$ for all $\hat{a}\in (M^*)^{m-1}$. 
The relation $\lesssim_{\hat{a}}$ pushes forward to a linear order on $A_{\Omega_{2}(\hat{a})}$, which we denote by $\leq_{\hat{a}}$. For $(\hat{a},x)\in  (M^{*})^{m-1} \times M^{*}$ let $V_{(\hat{a},x)}$ be the set of $b\in (M^{*})^{n}$ such that $f^{*}(\hat{a},b)\leqslant x$.   Then $V_{(\hat{a},x)}/{\sim_{\hat{a}}}$ is easily seen to be a cut in the definable linear order $(A_{\Omega_{2}(\hat{a})},{\leq_{\hat{a}}})$, and hence definable (in $\cM)$, by Proposition~\ref{prp:Janak, consequence}. In fact, by Proposition~\ref{prp:Janak, consequence, uniform}, there is a definable $D\subseteq M^p\times A$, for some $p$, and a map $\Omega_3\colon   (M^{*})^{n-1}\times (M^{*})\longrightarrow M^p$,
definable in $(\cN,\cM)$, such that 
$$[b]_{\sim_{\hat{a}}}\in V_{(\hat{a},x)}/{\sim_{\hat{a}}}\quad\Longleftrightarrow\quad b\in D_{\Omega_{3}(\hat{a},x)}.$$ Hence
\begin{align*}
\cM^{*}\models\d(a;b)&\quad\Longleftrightarrow\quad (\hat{a},b)\in X^*\land \big[f^*(\hat{a},b)\leq a_{m}\big]\\ &\quad \Longleftrightarrow \quad \big[b\in B_{\Omega_1(\hat{a})}\big] \land \big[b\in D_{\Omega_3(\hat{a},x)}\big]
\end{align*}
Therefore $\tp(a|M)\!\upharpoonright\!\d$ is definable in the way indicated in the proposition. \qed


\begin{thebibliography}{10}

\bibitem{van den Dries 1}
Lou van den Dries, {\it Limit sets in o-minimal structures}, in: M. Edmundo et al. (eds.), {\it Proceedings of the RAAG Summer School Lisbon 2003: O-minimal Structures,}  172--215, Lecture Notes in Real Algebraic and Analytic Geometry,
Cuvillier Verlag, G\"ottingen, 2005.

\bibitem{van den Dries 2} \bysame, {\it Tame Topology and O-minimal Structures,} Cambridge Univ. Press, New York, 1998.


\bibitem{PPR} Pantelis Eleftheriou, Ya'acov Peterzil, and Janak Ramakrishnan, {\it Interpretable groups are definable,}
preprint.


\bibitem{Marker-Steinhorn} Dave Marker and Charles Steinhorn, {\it Definable types in o-minimal theories,} J. Symbolic Logic {\bf 59} (1994), no. 1, 185--198. 

\bibitem{Onshuus-Steinhorn} Alf Onshuus and Charles Steinhorn, {\it On linearly ordered structures of finite rank,}
 J. Math. Log. {\bf 9} (2009), no. 2, 201--239.

\bibitem{Ramakrishnan} Janak Ramakrishnan, {\it Definable linear orders definably embed into
lexicographic orders in o-minimal structures,} preprint.

\bibitem{CS} Artem Chernikov and Pierre Simon, {\it Externally definable sets and dependent pairs,}\/~II, preprint (2012).

\bibitem{Tressl} Marcus Tressl, {\it Valuation theoretic content of the Marker-Steinhorn Theorem,} J. Symbolic Logic {\bf 69} (2004), no. 1, 91--93.

\end{thebibliography}
\end{document}